\documentclass[12pt]{amsart}

\usepackage{amsmath,amssymb,amscd,amsxtra,upref}
\usepackage{latexsym}
\usepackage[dvips]{graphics,epsfig}
\usepackage{enumerate}
\usepackage[colorlinks=true, pdfstartview=FitV, linkcolor=blue, citecolor=blue, urlcolor=blue]{hyperref}
\usepackage{color}
\usepackage{bbm}
\usepackage{mathabx}
\allowdisplaybreaks
\DeclareFontFamily{U}{mathx}{\hyphenchar\font45}
\DeclareFontShape{U}{mathx}{m}{n}{
      <5> <6> <7> <8> <9> <10>
      <10.95> <12> <14.4> <17.28> <20.74> <24.88>
      mathx10
      }{}
\DeclareSymbolFont{mathx}{U}{mathx}{m}{n}
\DeclareFontSubstitution{U}{mathx}{m}{n}
\DeclareMathAccent{\widecheck}{0}{mathx}{"71}
\DeclareMathAccent{\wideparen}{0}{mathx}{"75}

 \makeatletter
\newtheorem*{rep@theorem}{\rep@title}
\newcommand{\newreptheorem}[2]{%
\newenvironment{rep#1}[1]{%
 \def\rep@title{#2 \ref*{##1}}%
 \begin{rep@theorem}}%
 {\end{rep@theorem}}}
\makeatother

\newtheorem{theorem}{Theorem}[section]
\newtheorem{lemma}[theorem]{Lemma}

\newtheorem{corollary}[theorem]{Corollary}

\newreptheorem{theorem}{Theorem}

\theoremstyle{remark}
\newtheorem{dfn}{Definition}[section]

\newcommand{\R}{\mathbb{R}}
\newcommand{\Rn}{\mathbb{R}^n}
\newcommand{\Rnn}{\mathbb{R}^{2n}}
\newcommand{\Ha}{\mathcal{H}}
\newcommand{\Hn}{\mathbb{H}^n}
\newcommand{\He}{\mathbb{H}}
\newcommand{\Sw}{\mathcal{S}}

\newcommand{\Cn}{\mathbb{C}^n}
\newcommand{\gnm}{\gamma_{n,m}}
\newcommand{\munm}{\mu_{2n,m}}
\newcommand{\Gh}{G_h(2n,m)}
\newcommand{\V}{\mathbb{V}}
\newcommand{\Ma}{\mathcal{M}}
\newcommand{\muhat}{\widehat{\mu}}
\newcommand{\nuhat}{\widehat{\nu}}
\newcommand{\Vp}{V^{\perp}}

\thanks{
{\bf Acknowledgments} The author would like to thank professor Jeremy T. Tyson for his guidance and helpful insights during this project. The author would also like to thank the referee for helpful comments and corrections that helped significantly improve this manuscript. This project was supported by NSF Grant DMS-1600650 ``Mappings and measures in sub-Riemannian and metric spaces."
}

\begin{document}





\address{Fernando Rom\'{a}n-Garc\'{i}a, Department of Mathematics, 250 Atgeld Hall, University of Illinois Urbana Champaign,
Urbana, Il 61801, USA.} \email{romanga2@illinois.edu}

\title[Isotropic projections and slices]
{Intersection of projections and slicing theorems for the isotropic Grassmannian and the Heisenberg group}
\author{Fernando Rom\'{a}n-Garc\'{i}a}

\begin{abstract}
This paper studies the Hausdorff dimension of the intersection of isotropic projections of subsets of $\Rnn$, as well as dimension of intersections of sets with isotropic planes. It is shown that if $A$ and $B$ are Borel subsets of $\Rnn$ of dimension greater than m, then for a positive measure set of isotropic m-planes, the intersection of the images of $A$ and $B$ under orthogonal projections onto these planes have positive Hausdorff $m$-measure.
In addition, if  $A$ is a measurable set of Hausdorff dimension greater than $m$, then there is a set $B\subset\Rnn$ with $\dim B\leq m$ such that for all $x\in\Rnn\setminus B$ there is a positive measure set of isotropic m-planes for which the  translate by $x$ of the orthogonal complement of each such plane, intersects $A$ on a set of dimension $\dim A-m$. 
These results are then applied to obtain analogous results on the $n^{th}$ Heisenberg group.
\end{abstract}

\maketitle

\section{Introduction}
For $m\leq n$, denote by $G(n,m)$ the Grassmannian of m-dimensional subspaces of $\Rn$. $G(n,m)$ is endowed with a unique $\mathcal{O}(n)$-invariant probability measure $\gnm$. These spaces play an important role in the integral geometry of Euclidean space. A lot of work has been done concerning the effect of orthogonal projection maps, $P_V:\Rn\to V$, on the dimension and measure of sets. One of the most influential results dates back to J. Marstrand \cite{Marstrand} who proved that for a Borel set $A\subset\R^2$, $\dim P_VA=\min\{\dim A, 1\}$ for $\gamma_{2,1}$-a.e. $V\in G(2,1)$. This was later generalized to higher dimensions by P. Mattila in \cite{Mattila3} and later work. The most general result, including the Besicovitch-Federer characterization of unrectifiability (\cite{Bes1}, \cite{Federer}), can be stated in the following encompassing theorem.
\begin{theorem}\label{PROJ}
Let $A\subset\Rn$ be a Borel set of dimension $s$. 
\begin{enumerate}
\item If $s\leq m$, $\dim P_VA = s$ for $\gnm$-a.e. $V\in G(n,m)$.
\item If $s> m$, $\Ha^m( P_VA) > 0$ for $\gnm$-a.e. $V\in G(n,m)$.
\item If $s> 2m$, $Int( P_VA) \neq\varnothing$ for $\gnm$-a.e. $V\in G(n,m)$.
\end{enumerate} 
Moreover, in the case where $s=m$, and with the added hypotesis that $\Ha^m(A)<\infty$ ,  $\Ha^m( P_VA) = 0$ for $\gnm$-a.e. $V\in G(n,m)$ if, and only if, $A$ is purely $m$-unrectifiable. 
\end{theorem}
 Several other related problems have been studied with great success. For instance problems concerning intersection of projections of sets were studied by Mattila and Orponen in \cite{MattilaOrponen}, where they proved the following,
 
 \begin{theorem}\label{IntProj}
Let $A,B\subset \Rn$ be Borel sets,
\begin{enumerate}
\item If $\dim A>m$ and $\dim B>m$, then
\[\gnm (V\in G(n,m):\Ha^m(P_V A \cap P_V B)>0)>0.\]
\item If $\dim A>2m$ and $\dim B>2m$, then
\[\gnm (V\in G(n,m): Int(P_V A \cap P_V B)\neq\varnothing)>0.\]
\item If $\dim A>m,\ \dim B\leq m$, and $\dim A+\dim B>2m$, then for all $\epsilon>0$,
\[\gnm(V\in G(n,m): \dim [P_VA\cap  P_VB]>\dim B-\epsilon)>0.\]
\end{enumerate}
\end{theorem}
 
Another result that will be relevant here is the almost sure dimension estimates of planar slices of sets in $\Rn$. This was, once again, considered in the plane by Marstrand in \cite{Marstrand} and generalized to higher dimensions by Mattila in \cite{Mattila3}. Recently there have been further work in this area. For instance in \cite{MattilaOrponen}, P. Mattila and T. Orponen showed the following, 

 \begin{theorem}\label{SLICE}
 Let $m<s\leq n$ and let $A\subset\Rn$ be a Borel set with $0<\Ha^s(A)<\infty$. Then there is a set $B$ with $\dim B\leq m$ and such that for every $x\in\Rn\setminus B$,
 \[\gnm(V\in G(n,m): \dim [A\cap(\Vp+x) ]=s-m)>0.\]
 \end{theorem}
This paper will study integral geometric properties of the isotropic Grassmannian, with the goal of obtaining a series of results analogous to what is known for the standard Grassmannian in $\Rn$.

Denote by $\omega$ the standard symplectic form in $\Rnn$, $\omega((x,y),(u,v))=x\cdot v-y\cdot u$, and  for $0<m\leq n$ consider the collection $\Gh =\{ V\in G(2n,m): \omega\lfloor V\equiv 0\}$. This collection is known as the isotropic Grassmannain of m dimensional subspaces of $\Rn$. The spaces $\Gh$ are very well known and have been heavily studied for quite some time. In particular the space $G_{h}(2n,n)$ is known as the Lagrangian Grasmannian and is central in the study of symplectic geometry and Hamiltonian mechanics. The reader is referred \cite[Section 2]{BFMT} for a concise and clear summary of the main properties of the isotropic Grassmannian. Here stated, without proof or deep discussion, are some of the properties that will be relevant in this manuscript. 

The space $\Gh$  is a smooth submanifold  of $G(2n,m)$ with positive codimension. Indeed, $\dim(\Gh)=2nm-\frac{m(3m-1)}{2}=:\eta_m^n$. This tells us that $\gnm(\Gh)=0$. Hence, a result analogous to Theorem \ref{PROJ} cannot be obtained by simply applying the theorem to planes in $\Gh$, instead a natural measure on $\Gh$ is needed. This can be obtained by noting that the unitary group, $\mathcal{U}(n)$, acts smoothly and transitively on $\Gh$. Therefore $\Gh$ inherits a unique $\mathcal{U}(n)$-invariant probability measure that is here denote by $\munm$. Once again, it is important to note that $\munm$ is singular with respect to $\gamma_{2n,m}$. Nevertheless, it is possible to obtain isotropic analogues of many of the aforementioned results about $G(n,m)$. For instance, R. Hovila showed in \cite{Hovila} that for every $1\leq m\leq n$, the family of isotropic projections, $\{P_V\}_{V\in\Gh}$, is transversal in the sense of Peres and Schlag \cite{PeresSchlag}. This gives, as a corollary, the following
\begin{theorem}\label{ISOPROJ}
Let $A\subset\Rnn$ be a Borel set of dimension $s$ with $0<\Ha^s(A)<\infty$. Then, 
\begin{enumerate}
\item If $s\leq m$, $\dim P_VA = s$ for $\munm$-a.e. $V\in \Gh$.
\item If $s> m$, $\Ha^m( P_VA) > 0$ for $\munm$-a.e. $V\in \Gh$.
\item If $s>2m$ then, $Int(P_V(A)\neq\varnothing, \text{ for } \munm-a.e. V\in\Gh$.
\end{enumerate} 
In the case when $s=m$, $\Ha^m( P_VA) = 0$ for $\munm$-a.e. $V\in \Gh$ if, and only if, $A$ is purely $m$-unrectifiable. 
\end{theorem}
Parts (1) and (2) were originally proven in \cite{BFMT}. The later proof using transversality also provides dimension estimates for the subset of isotropic planes where the theorem fails. \\

In this work, the results will be analogous statements to Theorems  \ref{IntProj} and \ref{SLICE}. The following are the main results.
\begin{theorem}\label{isoprojint}
Let $A,B\subset \Rnn$ be Borel sets, and for $V\in \Gh$ let $P_V$ be the orthogonal projection onto $V$.
\begin{enumerate}
\item If $\dim A>m$ and $\dim B>m$, then
\[\munm (V\in \Gh:\Ha^m(P_V A \cap P_V B)>0)>0.\]
\item If $\dim A>2m$ and $\dim B>2m$, then
\[\munm (V\in \Gh: Int(P_V A \cap P_V B)\neq\varnothing)>0.\]
\item If $\dim A>m,\ \dim B\leq m$, and $\dim A+\dim B>2m$, then for all $\epsilon>0$,
\[\munm(V\in\Gh: \dim[ P_VA\cap  P_VB]>\dim B-\epsilon)>0.\]
\end{enumerate}
\end{theorem}
And,
\begin{theorem}\label{isoslice}
 Let $A\subset\Rnn$ be a Borel set such that for some $m<s\leq n$, $0<\Ha^s(A)<\infty$. Then, there is a Borel set $B\subset\Rnn$ with $\dim B\leq m$, such that for all $x\in\Rnn\setminus B$,
 \[\munm(V\in \Gh: \dim [A\cap(\Vp+x) ]=s-m)>0.\]
\end{theorem}

Since the $\{P_V:V\in\Gh\}$ forms a transversal family, it is not unexpected that results concerning the family of all orthogonal projections $\{P_V:V\in G(n,m)\}$, carry over to the isotropic case. This work is a confirmation of this intuition, and many of the proof are very simple adaptations of the arguments used in \cite{MattilaOrponen}. These results, however, let us obtain similar results in the context of the Heisenberg group $\Hn$. As noted by P. Mattila, R. Serapioni and F. Serra-Cassano in \cite{MaSeFr} and later used by Z. Balogh, K. F\"assler, P. Mattila, and J.T. Tyson in  \cite{BFMT}, the manifolds $\Gh$ play a crucial role in the horizontal geometry of the Heisenberg group. It is therefore not surprising that the results obtained for $\Gh$ in this paper have direct analogues in the Heisenberg group. For instance,

\begin{theorem}\label{BFMT}
Let $A,B\subset\Hn$ be Borel sets, let $\V$ denote the horizontal subgroup corresponding to the isotropic plane $V\in\Gh$, and let $P_{\V}$ denote the homogeneous projection onto $\V$.
\begin{enumerate}
\item If $\dim A,\dim B>m+2$ then,
\[\munm(V\in\Gh:\Ha^m(P_{\V}A\cap P_{\V}B)>0)>0.\]
\item If $\dim A,\dim B>2m+2$ then,
\[\munm(V\in \Gh:Int(P_{\V}A\cap P_{\V}B)\neq\varnothing)>0.\]
\item If  $\dim A> m+2,\ \dim B\leq m+2$ but $\dim A + \dim B>2m+4$, then
\[ \munm(V\in \Gh: \dim[ P_\V A\cap  P_\V B]>\dim B -2 -\epsilon)>0.\]
\end{enumerate}
\end{theorem}

In  \cite{BFMT}, the authors were able to obtain a slicing theorem for $\Hn$ along the lines of Theorem \ref{SLICE}. In their paper, they consider slices by right cosets of vertical planes and obtained,

\begin{theorem}\label{FubinatedHslice}
Let $A\subset\Hn$ be a Borel set with $\dim A=s>m+2$ and such that $0<\Ha^{s}(A)<\infty$. Then
\[\Ha^m(u\in\V:\dim [A\cap(\V^{\perp}*u)]=s-m)>0,\text{ for }\munm-a.e. V\in\Gh.\]
\end{theorem}

Using the results for $\Gh$ obtained in this paper, it is possible to prove very similar result to this last theorem and further obtain a dimension bound for the set of exceptions to this slicing result.

\begin{theorem}\label{sliceinh}
Let $s\in\R$ be such that $m+2<s\leq 2n+2$, and $A\subset\Hn$ be a Borel set with $0<\Ha^{s}(A)<\infty$. Then, there is a Borel set $B\subset\Hn$ with $\dim B\leq m+2$, such that for all $p\in\Hn\setminus B$,
\begin{equation}\label{horislice}
\munm(V\in\Gh:\dim [A\cap(\V^{\perp}*p)]=s-m)>0.
\end{equation}
\end{theorem}


\section{Preliminaries}


\subsection{Hausdorff measure and dimension}
Given a metric space $(X,d)$ and a subset $A\subset X$ we will denote by $\Ha^s(A)$ the Hausdorff $s$-measure of A. That is,
\[\Ha^s(A)=\liminf_{\delta\to 0}\{\sum \text{diam}(E_j)^s: A\subset\bigcup_jE_j, \text{diam}(E_j)\leq\delta\}.\]
The Hausdorff dimension of $A$ is defined as  $\dim A=\inf\{s>0:\Ha^s(A)=0\}=\sup\{s>0:\Ha^s(A)=\infty \}$.
Aside from the definition, there are other computational tools to estimate Hausdorff dimension of sets. One of the most useful ones is Frostman's lemma, proven by Otto Frostman as part of his dissertation.

\begin{dfn}
A positive Borel measure $\mu$ on $A\subset X$ is said to be a Frostman $s$-measure if $\mu(B(a,r))\leq r^s$, for all $a\in A$ and $r>0$.
\end{dfn}
\begin{theorem}[Frostman lemma]
Let $(X,d)$ be a complete, separable metric space and $A$ a Suslin subset of $X$. Then the following are equivalent,
\begin{enumerate}
\item $\Ha^s(A)>0$.
\item There is a Frostman $s$-measure $\mu$ supported on $A$ such that $\mu(A)>0$.
\end{enumerate}
\end{theorem}
A constructive proof of Frostman's lemma for compact sets is given in \cite[Theorem 8.8]{mattila2}. The more general case of Suslin sets was originally treated by Carleson in \cite{Carleson}. Frostman's lemma can also be used to obtain another powerful tool which allows the computation of Hausdorff dimension of a set in terms of energies of measures supported on that set. This is sometimes referred to as the ``energy version" of Frostman's lemma.
\begin{corollary}\label{FrostEner}
If $A$ is a Suslin subset of $(X,d)$ and we denote by $\Ma(A)$ the set of all compactly supported Radon measures on A, then

\[\dim A=\sup\{s\geq0:\exists \mu\in\Ma(A),\  s.t.\ I_s(\mu)<\infty \}.\]
Here \[I_s(\mu)=\int\int d(x,y)^{-s}d\mu(y)d\mu(x)\] denotes the $s$-energy of the measure $\mu$.
\end{corollary}
This approach was first introduced in $\R$ by R. Kaufman in \cite{Kaufman}. A simple proof of this is given in \cite[Theorem 2.8]{Mattila1}.


\subsection{Fourier transform of measures}
In the same way as for $L^1$ functions, one can define the Fourier transform of a finite measure as $\muhat(\xi)=\int_{\Rn}e^{-2\pi ix\cdot\xi}d\mu$. Note that with this definition, $\muhat$ is a bounded Lipschitz function. Convolution of functions and measures are defined as 
\begin{equation}
(f*\mu)(x):=\int f(x-y)d\mu(y),
\end{equation}
whenever the integral is defined. More generally, convolution of finite measures is defined as another finite measure, $\mu\ast\nu$, given by
\begin{equation}\label{MeasConv}
\int \varphi(z)d(\mu\ast\nu)(z):=\iint\varphi(x+y)d\mu(x)d\nu(y),
\end{equation}
for all non-negative, continuous bounded functions $\varphi$ (see \cite[pp.16]{mattila2}).

The following properties hold,

\begin{lemma}\label{prop}\text{ }

\begin{enumerate}
 \item $\widehat{f*\mu}=\widehat{f}\muhat, \ \text{for}\  f\in L^1(\Rn)$.
 \item $\widehat{\mu\ast\nu}=\muhat\widehat{\nu}$ for $\mu,\nu\in\Ma(\Rn).$
 \item$ \int fd\mu = \int \widehat{f}\widehat{\mu},\ \forall\ f\in\Sw(\Rn)$.
 \item If $\muhat\in L^2(\Rn)$, then $\mu=fdx$ for some $f\in L^2(\Rn).$
 \item  If $\muhat\in L^1(\Rn)$, then $\mu=gdx$ for some $g\in \mathcal{C}(\Rn).$
\end{enumerate}
\end{lemma}
 One of the main connection between Fourier analysis and Hausdorff dimension comes from using the Fourier transform to estimate energy integrals. For $0< s<n$, we denote by $R_s$ the kernel of the Riesz potential, $R_s(x)=|x|^{-s}$. The energy of a measure $\mu$ can be written as $I_s(\mu)=\int R_s*\mu d\mu$. Using the distributional Fourier transform and its properties, one can compute $\widehat{R}_s=\kappa(n,s)R_{n-s}$, where $\kappa$ is an universal constant depending on $s$ and $n$. Using this results together, smooth approximations to the measure $\mu$ and the properties of the Fourier transform discussed above, one can justify the following formula,
 \begin{equation}\label{Fenergy}
 I_s(\mu)=\int R_s*\mu d\mu = \kappa(n,s)\int R_{n-s}|\muhat|^2=\kappa(n,s)\int|\muhat(x)|^2|x|^{s-n}dx.
 \end{equation}
 
 A proof of this can be found on \cite[pp. 38]{Mattila1}. \\
 Finally, we will also make use of the mutual energy of measures. For compactly supported measures $\mu$ and $\nu$ we define the mutual $s$-energy, with $0\leq s\leq n$, as  \[I_s(\mu,\nu)=\int\int|x-y|^{-s}d\nu(y)d\mu(x).\] 
 Just as with energies of measures, it is convenient to write the mutual energy as an integral on the frequency domain. If $I_{\alpha}(\mu)<\infty,\ I_{\beta}(\nu)<\infty$ and $s=\frac{\alpha+\beta}{2}$, by Holder's inequality we have,
 \begin{equation}\label{holderen}
 \int_{\Rn}|\muhat(x)\overline{\nuhat}(x)||x|^{s-n}dx\leq\left(\int|\muhat(x)|^2|x|^{\alpha-n} \right)^{1/2}\left(\int|\widehat{\nu}(x)|^2|x|^{\beta-n} \right)^{1/2}\lesssim(I_{\alpha}(\mu)I_{\beta}(\nu))^{1/2}.
 \end{equation}
 Therefore, if $\{\psi_{\delta}\}_{\delta>0}\subset\Sw(\Rn)$ is an approximation to the identity, and we take $\mu_{\delta}=\psi_{\delta}*\mu,\ \nu_{\delta}=\psi_{\delta}*\nu$, the fact that $\widehat{\psi}_{\delta}$ is uniformly bounded tells us that \[|\muhat_{\delta}\overline{\nuhat}_{\delta}||\cdot|^{s-n}=|\muhat\overline{\nuhat}||\widehat{\psi}_{\delta}|^{2}|\cdot|^{s-n},\]
 is dominated by $|\muhat\overline{\nuhat}||\cdot|^{s-n}.$ This validates the formula
 \begin{equation}\label{JointEner}
 I_s(\mu,\nu)=\kappa(n,s)\int \muhat(x)\widebar{\widehat{\nu}}(x)|x|^{s-n}dx,
 \end{equation}
 for $s\leq\frac{\alpha+\beta}{2}$. Also note that since $\mu$ and $\nu$ are positive measures the mutual energy is positive.
 

\section{Orthogonal projections onto isotropic planes}
We begin this section by establishing already known results about dimension distortion by isotropic projections. The most natural statement is that analogous to Theorem \ref{PROJ} and we state it in its strongest form as follows.
\begin{theorem}\label{IsoProjExcep}
If $A\subset\Rnn$ is a Borel set with $\dim A=\alpha$, then
\begin{enumerate}
\item If $\alpha\leq m$ and $t\in(0,\alpha]$,  $\dim\{V\in\Gh:\dim(P_VA)<t\}\leq \eta^{n}_{m} +t-\alpha$
\item If $\alpha> m$,  $\dim\{V\in\Gh:\Ha^m(P_VA)=0\}\leq \eta^{n}_{m} +m-\alpha$
\item If $\alpha> 2m$,  $\dim\{V\in\Gh:Int(P_VA)\neq\varnothing\}\leq \eta^{n}_{m} +2m-\alpha$.
\end{enumerate}

Furthermore, if $\alpha=m,$ and $\Ha^m(A)<\infty$ then $\Ha^m(P_VA)=0$ for $\munm$-a.e. $V\in\Gh$ if, and only if, $A$ is purely m-unrectifiable.
\end{theorem}
This was proven by proven by R. Hovila in \cite{Hovila}, and implies the almost everywhere statements of Theorem \ref{ISOPROJ}.

With these general result about projections onto isotropic planes in hand, we now turn our attention to intersections of such projections. We would like to establish a result analogous to Theorem \ref{IntProj}. In order to establish such a result we will need an isotropic version of the Grassmannian disintegration fromula in $\Rnn$. This is precisely the content of the following 

\begin{lemma}[Isotropic disintegration formula]\label{desintegrate}
There exists a positive constant $c=c(n,m)$ such that, for all $f\in L^1(\Rnn)$,
\[\int_{\Rnn}f(x)dx=c\int_{\Gh}\int_V|u|^{2n-m}f(u)d\Ha^m(u)d\munm(V).\]
\end{lemma}

\begin{proof}
Using spherical coordinates on $V$ we can compute the integral on the left as follows, 

\[\int_{\Gh}\int_0^{\infty}\int_{V\cap\mathbb{S}^{2n-1}}r^{2n-m}f(rv)d\sigma^{m-1}(v)r^{m-1}drd\munm(V)\]
\[=\int_{0}^{\infty}r^{2n-1}\int_{\Gh}\int_{V\cap\mathbb{S}^{2n-1}}f(rv)d\sigma^{m-1}(v)d\munm(V)dr.\]
Now we take a closer look at the inner double-integral. We know that $U(n)$ acts transitively on $\mathbb{S}^{2n-1}$, therefore, up to multiplication by a constant, there is a unique $U(n)$-invariant measure on $\mathbb{S}^{2n-1}$. Since $\sigma^{2n-1}$ is $O(n)$ invariant, it is in particular $U(n)$ invariant. Hence any $U(n)$-invariant measure on $\mathbb{S}^{2n-1}$ must be a constant multiple of the surface measure $\sigma^{2n-1}$. Now, for a function $\varphi$ on $\mathbb{S}^{2n-1}$, one can check that the measure on $\mathbb{S}^{2n-1}$ given by, $\int_{\Gh}\int_{V\cap\mathbb{S}^{2n-1}}\varphi(v)d\sigma^{m-1}(v)d\munm(V),$ is $U(n)$-invariant. So as noted before, there exist a constant $c=c(n,m)$ such that 
\[\int_{\Gh}\int_{V\cap\mathbb{S}^{2n-1}}\varphi(v)d\sigma^{m-1}(v)d\munm(V)=c\int_{\mathbb{S}^{2n-1}}\varphi(v)d\sigma^{2n-1}(v).\]
Going back to our overall integral, we now have,
\[I=\int_{0}^{\infty}r^{2n-1}\int_{\Gh}\int_{V\cap\mathbb{S}^{2n-1}}f(rv)d\sigma^{m-1}(v)d\munm(V)dr\]
\[=c\int_0^{\infty}r^{2n-1}\int_{\mathbb{S}^{2n-1}}f(rv)d\sigma^{2n-1}(v)dr=c\int_{\Rnn}f(x)dx.\]

\end{proof}

One application of this formula is the following lemma, which will be useful later,
\begin{lemma}\label{FrostAbs}
Let $m<s\leq 2n$ and $\mu\in\Ma(\Rnn)$ be a measure such that, for some $C>0$, \[\mu(B(x,r))\leq Cr^{s}\text{ for all } x\in\Rnn,\ r>0.\] Then, $P_{V\#}\mu\ll\Ha^m$ for $\munm-$almost every $V\in\Gh$.
\end{lemma}
\begin{proof}
Using \cite[Theorem 1.15]{mattila2}, for any $\sigma< s$ we have 
\begin{align*}
I_{\sigma}(\mu)=&\int\int |x-y|^{-\sigma}d\mu(y)d\mu(x)\\
&=\int \int_0^{\infty}\mu(B(x,r^{-\frac{1}{\sigma}}))drd\mu(x)\\
&=\sigma\int\int_0^{\infty}\mu(B(x,u))u^{-\sigma-1}dud\mu(x)\\
&\leq \sigma\mu(\Rnn)\left[C\int_0^1u^{s-\sigma-1}du+\mu(\Rnn)\int_1^{\infty}u^{-\sigma-1}du\right]<\infty.
\end{align*}
In particular $I_m(\mu)<\infty$. By Lemma \ref{desintegrate}, and using the fact that for any plane $V$ and $v\in V$, $\widehat{P_{V\#}\mu}(v)=\widehat{\mu}(v)$, we compute
\begin{align*}
\int_{\Gh} &\int_V|\widehat{P_{V\#}\mu}(v)|^2d\Ha^m(v)d\munm(V)\\
&=\int_{\Gh} \int_V|\widehat{\mu}(v)|^2d\Ha^m(v)d\munm(V)\\
&=c(n,m)^{-1}\int_{\Rnn}|x|^{m-2n}|\widehat{\mu}(x)|^{2}dx\\
&=c(n,m)^{-1}I_m(\mu)<\infty.
\end{align*}
That is, $\widehat{P_{V\#}\mu}\in L^{2}(V)$ for $\munm-$ almost every $V\in\Gh$. It follows from \cite[Theorem 3.3]{Mattila1}, that $P_{V\#}\mu(v)\ll\Ha^m$ with density in $L^2(V)$, for $\munm-$almost every $V\in\Gh$.
\end{proof}

Part (1) of Theorem \ref{isoprojint} is proven next.

\begin{theorem}\label{caseA}
If $A,B\subset\Rnn$ are Borel sets with $\dim A>m$ and $\dim B>m$, then 
\[\munm(V\in\Gh: \Ha^m(P_V A\cap P_V B)>0)>0\]

\end{theorem}
\begin{proof}
Pick $\mu\in\Ma(A)$ and $\nu\in\Ma(B)$ both with finite $m$ energy. As seen in the proof of Lemma \ref{FrostAbs} we know that for $\munm$-a.e. V, the measures $P_{V\#}\mu$ and $P_{V\#}\nu$ are $L^2(V)$ functions, we denote these functions by $\mu_V$, and $\nu_V$. By H\"{o}lder's inequality, the product $\mu_V\nu_V$ is in $L^1(V)$ for  $\munm$-a.e. $V$. Now we note that if $\int_V\mu_V\nu_Vd\Ha^m>0$, then $\Ha^m(spt(\mu_V)\cap spt(\nu_V))>0$ because if $a\in V$ is such that $\mu_V(a)\nu_V(a)>0$ then $a\in spt(\mu_V)\cap spt(\nu_V)$. Thus, we just need to show that $\int_V\mu_V\nu_Vd\Ha^m>0$ for $V$ ranging over a $\munm$-positive measure subset of $\Gh$.

To this end, we note that for $a\in V$, $\widehat{\mu}_V(a)=\widehat{\mu}(a)$, and similarly for $\nu$. Applying Plancherel's theorem on $V$, and Lemma \ref{desintegrate} we get,\\

\begin{equation}\label{L1comp}
\begin{split}
&c\int_{\Gh}\int_V\mu_V(a)\nu_V(a)d\Ha^m(a)d\munm(V)\\
&= c\int_{\Gh}\int_V\widehat{\mu}_V(a)\overline{\widehat{\nu}}_V(a)d\Ha^m(a)d\munm(V) \\
&= c\int_{\Gh}\int_V\widehat{\mu}(a)\overline{\widehat{\nu}}(a)d\Ha^m(a)d\munm(V)\\
&=\int_{\Rnn}|x|^{m-2n}\widehat{\mu}(x)\overline{\widehat{\nu}}(x)dx \\ 
&= I_m(\mu,\nu)>0.
\end{split}
\end{equation}
\end{proof}

To prove part (2), a statement at the level of measures is proven first. This will directly imply the desired result.

\begin{theorem}\label{isoprojmeasures}
Assume $\mu,\nu$ are compactly supported Radon measures in $\Rnn$ with finite $s$ and $t$ energy respectively.
\begin{enumerate}
\item If $s+t= 2m$, $\munm(V\in\Gh:P_V(spt \mu)\cap P_V(spt \nu)\neq\varnothing)>0$.
\item If  $s>2m,$ and $t>2m$, then $\munm(V\in\Gh:Int(P_V(spt \mu)\cap P_V(spt \nu))\neq\varnothing)>0$.
\end{enumerate}
\end{theorem}

\begin{proof}
\begin{enumerate}

\item 
Since $m=\frac{s+t}{2}$, we have that
\[\int_{\Rnn}|x|^{m-2n}|\muhat(x)\nuhat(x)|dx\lesssim (I_s(\mu)I_t(\nu))^{1/2}<\infty.\]
It follows from Lemma \ref{desintegrate}, together with the fact that, for $v\in V,\  P_{V\#}\mu(v)=\muhat(v)$ and $\widehat{P_{v\#}\nu(v)}=\nuhat(v)$, that 
\[\int_{\Gh}\int_V|\widebar{\widehat{P_{V\#}\mu}}\widehat{P_{V\#}\nu}| d\Ha^md\munm<\infty,\]
so that $\widebar{\widehat{P_{V\#}\mu}}\widehat{P_{V\#}\nu}\in L^1(V)$ for $\munm-$almost every $V$. For such $V$, consider the convolution measure $\widetilde{P_{V\#}\mu}\ast P_{V\#}\nu$ given, as in \eqref{MeasConv}, by

\[\int_V\varphi d(\widetilde{P_{V\#}\mu}\ast P_{V\#}\nu)=\iint_{V\times V}\varphi(v-u)dP_{V\#}\mu(u)dP_{V\#}\nu(v)\]
for any non-negative, bounded continuous function $\varphi$. 
It should be clear that $\widetilde{P_{V\#}\mu}\ast P_{V\#}\nu\in\Ma(V)$, so computing its Fourier transform renders
\[(\widetilde{P_{V\#}\mu}\ast P_{V\#}\nu)^{\widehat{}}=\widebar{\widehat{P_{V\#}\mu}}\widehat{P_{V\#}\nu}\in L^1(V).\]
By (4) in Lemma \ref{prop} (\cite[Theorem 3.4]{Mattila1}), it follows that $\widetilde{P_{V\#}\mu}\ast P_{V\#}\nu\ll\Ha^m$ with continuous density, call it $g_V$. Since $g_V=(\widebar{\widehat{P_{V\#}\mu}}\widehat{P_{V\#}\nu})^{\widecheck{}}$, we have that \[g_V(0)=\int_V\widebar{\widehat{P_{V\#}\mu}}\widehat{P_{V\#}\nu}d\Ha^m.\]
 By Lemma \ref{desintegrate} 
 \begin{align*}
 c&\int_{\Gh}\int_V\widebar{\widehat{P_{V\#}\mu}}(v)\widehat{P_{V\#}\nu}(v)d\Ha^{m}(v)d\munm (V)\\
 &=\int_{\Gh}\int_V\widebar{\muhat}(v)\widehat{\nu}(v)d\Ha^{m}(v)d\munm (V)\\
 &=\int_{\Rnn}|x|^{m-2n}\widebar{\muhat}(x)\nuhat(x)dx=I_{m}(\nu,\mu)>0.
 \end{align*}
It follows that $g_V(0)=\int_V\widebar{\widehat{P_{V\#}\mu}}\widehat{P_{V\#}\nu}d\Ha^m>0$ for a $\munm-$positive measure set of planes $V\in\Gh$. Since $spt(g_V)=spt(\widetilde{P_{V\#}\mu})+spt(P_{V\#}\nu)=spt(P_{V\#}\nu)-spt(P_{V\#}\mu)$, and $0\in spt(g_V)$, it follows that there is $u\in spt(P_{V\#}\mu)$ and $v\in spt(P_{V\#}\nu)$ such that $0=v-u$, or equivalently, $v=u$.
Hence, $spt(\mu_V)\cap spt(P_{V\#}\nu)\neq\varnothing$. This proves the claim.
 \bigskip

\item Now assume that $s,t>2m$. Then, by Theorem 3.4 in \cite{Mattila1}, both $P_{V\#}\mu$ and $P_{V\#}\nu$ have continuous densities $\mu_V$ and $\nu_V$ for  $\munm$-a.e. $V$. Since, by \eqref{L1comp} we know that for a $\munm$-positive set $\int_V\mu_V\nu_Vd\Ha^m>0$, it follows that for such $\munm$-positive set, $\mu_V\nu_V$ remains positive in some open subset of $V$. Since $spt(\mu_V\nu_V)\subset spt(\mu_V)\cap spt(\nu_V)$, the theorem follows.

\end{enumerate}
\end{proof}

As mentioned above, as a direct corollary of part (2) of Lemma \ref{isoprojmeasures}, we obtain part (2) of Theorem \ref{isoprojint}.

Now lets proceed with the proof of part (3) of Theorem \ref{isoprojint}.

\begin{proof}
We begin as usual, by choosing $s,t\in\R$ such that $m<s<\dim A$, $0<t< \dim B$ and $s+t>2m$. By Frostman's lemma, one can then pick $\mu \in \Ma(A)$ and $\nu\in\Ma(B)$ with finite $s$ and $t$ energies respectively. Since $s>m$ we have that for $\munm-a.e.\ V,$ $P_{V\#}\mu\ll d\Ha^m$ with density $\mu_V\in L^2(V)$. Following the same lines as the proof of part (1) of Theorem \ref{isoprojmeasures}, we aim to finds a family of measures $\rho_V\in\Ma(P_V(A)\cap P_V(B))$, but this time we also require that $I_t(\rho_V)<\infty$. In principle, we would like to use the family of measures $\rho_V=\mu_VdP_{V\#}\nu$. However, a priori, we do not know that $\mu_V\in L^1(P_{V\#}\nu)$. Instead, let $\mu_{\delta}=\mu*\psi_\delta$ be the standard convolution approximation to $\mu$, where $\psi$ is smooth and compactly supported in $B(0,1)$. Using Plancherel's theorem and Lemma \ref{desintegrate} we have,

\begin{equation}\label{Second}
\begin{aligned}
&\int_{\Gh}\int_{V}P_{V\#}\mu_{\delta} dP_{V\#}\nu(v)d\munm\\
&=\kappa(n,m)\int|x|^{m-2n}\widehat{\mu_{\delta}}(x)\nuhat(x)dx.\\
\end{aligned}
\end{equation}
Now, as $\delta\to0$, $\widehat{\mu_{\delta}}=\muhat\widehat{\psi_{\delta}}\to\muhat\widehat{\psi}(0)=\muhat$. Hence the right hand side of (\ref{Second}) goes to $\kappa'(n,m)I_m(\mu,\nu)$. By the choice of $\mu$ and $\nu$, and since $s+t>2m$, we know $0<I_m(\mu,\nu)<\infty$. Hence, $\exists\ c,C>0$ such that $\forall$ $\delta>0$,
\begin{equation}\label{control}
c<\int\int P_{V\#}\mu_{\delta}(v) dP_{V\#}\nu(v)d\munm<C.
\end{equation}
The aim is now to show that $\mu_V\in L^1(P_{V\#}\nu)$ and that 
\begin{equation}\label{BoundforDCT}
\iint\mu_VdP_{V\#}\nu d\munm=\lim_{\delta\to0}\iint P_{V\#}\mu_{\delta}dP_{V\#}\nu d\munm.
\end{equation}
This, together with \eqref{control} would show that
\[\mu_VdP_{V\#}\nu\in\Ma(P_VA\cap P_VB).\]
To prove \eqref{BoundforDCT} we follow a similar argument to the one used in \cite{MattilaOrponen} for the analogous statement.

First not that since $m-2\frac{s-m}{2}=2m-s<t$, and $I_{2m-s}(P_{V\#}\nu)\lesssim I_t(P_{V\#}\nu)$, using Lemma \ref{desintegrate}
\begin{equation}\label{boundnu}
\begin{split}
\int I_{2m-s}(P_{V\#}\nu)d\munm&\lesssim\int I_t(P_{V\#}\nu)d\munm\\
&\lesssim\iint |v|^{m-t}|\widehat{P_{V\#}\nu}|^2d\Ha^md\munm\\
&\lesssim \int_{\Rnn} |x|^{2n-t}|\nuhat(x)|^2 dx\lesssim I_{t}(\nu)<\infty.
\end{split}
\end{equation}
Next, note that we can write $P_{V\#}\mu_{\delta}(v)=\psi_{\delta}^V*P_{V\#}\mu$, where $\psi^V(v)=\int_{\Vp}\psi(v+w)d\Ha^{2n-m}(w)$ for $v\in V$. 

Once again by Lemma \ref{desintegrate},

\[\iint |v|^{s-m}|\widehat{\mu_V}(v)|^2d\Ha^m(v) d\munm(V) =\kappa(n,m)\int|x|^{s-2n}|\muhat(x)|^2dx\lesssim I_s(\mu)<\infty\]
which tells us that for $\munm$ almost every V, $\mu_V$ is in the fractional Sobolev space $H^{\frac{s-m}{2}}(V)$, with 
\begin{equation}\label{boundmu}
\int_{\Gh}||\mu_V||^2_{H^{\frac{s-m}{2}}(V)}d\munm(V)<\infty.
\end{equation}

Therefore, taking $\alpha=\frac{s-m}{2}$ in \cite[Theorem 17.3]{Mattila1}, we conclude that for $\munm-$almost every $V$, the maximal function $M_VP_{V\#}\mu(v)=\sup_{\delta>0}|\psi_{\delta}^V*P_{V\#}\mu(v)|$ is in the space $L^1(P_{V\#}\nu)$ with
\[\int_VM_VP_{V\#}\mu(v)dP_{V\#}\nu(v)\lesssim [I_{s-2m}(P_{V\#}\nu)]^{1/2}||\mu_V||^2_{H^{\frac{s-m}{2}}(V)}.\]
By the dominated convergence theorem it follows that for $\munm-$almost every $V\in\Gh$ the sequence $P_{V\#}\mu_{\delta}$ converges to $f_V:=\mu_V\lfloor_{spt(P_{V\#}\nu)}$ in $L^1(P_{V\#}\nu)$.
That is to say, the sequence of function 
\[F_{\delta}(V)=\int_V P_{V\#}\mu_{\delta}(v)dP_{V\#}\nu(v)\]
converges to the function 
\[F(V)=\int_V f_V(v)dP_{V\#}\nu(v),\]
for $\munm-$almost every $V\in\Gh$. But Moreover, for such $V$
\[|F_{\delta}(V)|\lesssim [I_{s-2m}(P_{V\#}\nu)]^{1/2}||\mu_V||^2_{H^{\frac{s-m}{2}}(V)}.\]
So by Cauchy-Schwarz inequality, \eqref{boundnu}, and \eqref{boundmu}, 
\begin{multline*}
\int_{\Gh} [I_{s-2m}(P_{V\#}\nu)]^{1/2}||\mu_V||^2_{H^{\frac{s-m}{2}}(V)}d\munm(V)\\
\leq\left(\int I_{s-2m}(P_{V\#}\nu)d\munm \right)^{1/2}\left(||\mu_V||^2_{H^{\frac{s-m}{2}}(V)}d\munm \right)^{1/2}<\infty.
\end{multline*} 
Once more by dominated convergence this gives us $\eqref{BoundforDCT}$, and by \eqref{control} we get

\begin{equation}\label{control2}
c<\iint f_V(v)dP_{V\#}\nu(v)d\munm(V)=\iint P_{V\#}\mu_{\delta}(v) dP_{V\#}\nu(v)d\munm<C.
\end{equation}

With this, consider the measure $f_VP_{V\#}\nu$. By (\ref{control2}) we know that for $\munm$-positively many $V,\ f_V$ is positive and finite on a set of positive $P_{V\#}\nu$ measure. Therefore $f_VP_{V\#}\nu$ is a non-trivial measure supported on $spt(\mu_V)\cap spt(\nu_V)\subset P_V(A)\cap P_V(B)$. That is, $f_VP_{V\#}\nu\in\Ma(P_V(A)\cap P_V(B))$. For each such $V$, pick a large enough constant $C_V$ so that the measure $\mathbbm{1}_{\{f_V\leq C_V\}}f_VP_{V\#}\nu$ is still non-trivial, and so that $\mathbbm{1}_{\{f_V\leq C_V\}}f_VP_{V\#}\nu$ has finite t-energy. Since $t$ can be chosen arbitrarily close to $\dim B$ the claim follows. 
\end{proof}


\section{Intersections with isotropic planes}

Another problem, which is related to projection theorems, and that has been studied in Euclidean space by many, for instance P. Mattila in \cite{Mattila3}, Mattila and Orponen  in \cite{MattilaOrponen} among others, and even in non-Euclidean setting by Balogh, F\"assler, Mattila and Tyson in \cite{BFMT}, is that of planar slices of sets. We aim to study this problem in Euclidean space but restricted to slices by isotropic planes. The general question can be stated as follows, for $A\subset\Rnn$ with $\dim A=\alpha$ what can be said about the ``size" (i.e. measure and/or dimension) of the ``slice" $A\cap(\Vp+x)$ where $V\in\Gh$ and $x\in\Rnn$? Ideally one would hope to obtain a result along the lines of Theorem \ref{SLICE}. It is important to remark that by $\Vp$ it is meant the standard orthogonal complement of $V\in\Gh$, in particular, $\Vp$ need not be isotropic.

As discussed before, for $A\subset\Rnn$  with $\dim A=\alpha>m$ we have that $\Ha^m(P_VA)>0$ for $\munm$-a.e. $V\in\Gh$. This implies that the set $v\in V$ such that $A\cap P^{-1}_V(v)\neq\varnothing$, has positive $m$-measure  for almost all $V$.  Note that $P_V^{-1}(v)=\Vp+v$, this hints at the close relation between this question and the questions about projections. This is all encapsulated in Theorem  \ref{isoslice} which is recalled here,

\begin{reptheorem}{isoslice}
Let $A\subset\Rnn$ be such that for some $m<s\leq n$, $0<\Ha^s(A)<\infty$. Then there exists a Borel set $B\subset\Rnn$ with $\dim B \leq m$, such that for all $x\in\Rnn\setminus B$,
 \[\munm(V\in \Gh: \dim [A\cap(\Vp+x) ]=s-m)>0.\]
\end{reptheorem}

In order to prove this, we will use ``sliced" measures. For the sake of completeness the construction of sliced measures context of the isotropic Grassmannian is included. A detailed overview of these measures in the context of the standard Grassmannian is contained in \cite[Section 10.1]{mattila2}. Fix $V\in\Gh$. Note that since $V$ is isotropic, $m\leq n$ because the maximal dimension of an isotropic plane is half the dimension of the space. Therefore $\dim \Vp=2n-m\geq n$, in particular, unless $m=n$, $\Vp$ cannot be isotropic. Nevertheless, we endow $\Gh^{\perp}$ with a ``natural" measure, $\munm^{\perp}$, via $\munm^{\perp}(\Omega)=\munm(\Omega^{\perp})$, where $\Omega^{\perp}=\{\Vp:V\in\Omega\}$.
One property that will be of great importance for dimension estimates is the following measure bound,

\begin{equation}\label{mubydelta}
\munm(\{V:|P_V(x)|\leq\delta\})\leq c\delta^m|x|^{-m}
\end{equation}
A proof of this fact can be found in \cite{BFMT}.

Given a measure $\mu\in\Ma(\Rnn)$ and function $\varphi\in\mathcal{C}^{+}_{c}(\Rnn)$ one can define a new measure by $\mu_{\varphi}(A)=\int_A\varphi d\mu.$ We can see that $P_{V\#}\mu_{\varphi}$ is a measure in $\Ma(V)$ so by the differentiation theorem, the derivative

\[\mu_{\Vp+v}(\varphi)\\
:=\lim_{\delta\to0}(2\delta)^{-m}P_{V\#}\mu_{\varphi}(B(v,\delta))\\
=\lim_{\delta\to0}(2\delta)^{-m}\int_{\mathcal{N}(\Vp+v,\delta)}\varphi d\mu,
\]
exists and is finite for $\Ha^m$-a.e. $v\in V$. Here $\mathcal{N}(\Vp+v,\delta)$ is the $\delta$ neighborhood around the plane $\Vp+v$. 
One can check, after some work, that $\mu_{\Vp+v}$ defines a positive linear functional on $\mathcal{C}^{+}_{c}(\Rnn)$. Therefore, by Riesz representation theorem we can associate a positive Radon measure to $\mu_{\Vp+v}$ that we denote in the same way. As mentioned before, the reader is referred to section 10.1 of \cite{mattila2} for the details.  This measure is now supported on $(\Vp+v)\cap spt\mu$.

Given two Radon measures, $\lambda$ and $\gamma$, the Radon-Nikodym derivative $\frac{d\lambda}{d\gamma}$ satisfies \[\int_A\frac{d\lambda}{d\gamma}(x)d\gamma(x)\leq\lambda(A),\]
for all Borel sets $A$, and with equality whenever $\lambda\ll\gamma$ (Theorem 2.12, \cite{mattila2}). Therefore, for any Borel set $B\subset V$ the measure $\mu_{\Vp+v}$ satisfies,
\begin{equation}\label{bound}
\int_B\int g d\mu_{\Vp+v} d\Ha^m(v)\leq\int_{P^{-1}_V(B)}g d\mu,
\end{equation}
for any Borel function $g$, and with equality whenever $P_{V\#}\mu\ll\Ha^m$.
In this case, taking $g=\mathbbm{1}_{P_V^{-1}(B)}$, 
\begin{equation}\label{PB}
\mu(P_V^{-1}(B))=\int_{B}\mu_{\Vp+v}(\Rnn)d\Ha^m(v)
\end{equation}
However, denoting by $\mu_V$ the density of $P_{V\#}\mu$, the definition of the push-forward measure tells us that
\begin{equation}\label{push}
\mu(P_V^{-1}(B))=\int_B\mu_{V}(v)d\Ha^m(v).
\end{equation}
Since this is true for any Borel set $B\subset V$, for $\Ha^m-$almost every $v\in V$
\begin{equation}\label{ProjSlice}
\mu_{\Vp+v}(\Rnn)=\mu_V(v)
\end{equation}
In particular, if $P_{V\#}\mu\ll\Ha^m$ then,
\begin{equation}\label{equal}
\mu(\Rnn)=\int\mu_{\Vp+v}(\Rnn)d\Ha^m(v).
\end{equation}

Using Lemma \ref{desintegrate} one may bound the energies $I_{s-m}(\mu_{\Vp+v})$ by $I_{s}(\mu)$ as follows,

\begin{lemma}\label{10.7analogue}
There is a constant $c=c(n,m)$ s.t. 
\[\int_{\Gh}\int_VI_{s-m}(\mu_{\Vp+v})d\Ha^m(v)d\munm(V)\leq c I_s(\mu).\]
\end{lemma}
\begin{proof}
\begin{align*}
&\int_{\Gh}\int_VI_{s-m}(\mu_{\Vp+v})d\Ha^m(v)d\munm(V)\\
&\leq\liminf_{\delta\to0} (2\delta)^{-m}\int\int_V\int_{\Vp+v}\int_{\mathcal{N}(\Vp+v,\delta)}|x-y|^{m-s}d\mu(x)d\mu_{\Vp+v}(y)d\Ha^m(v)d\munm(V)\\
&=\liminf_{\delta\to0} (2\delta)^{-m}\int\int_V\int_{\mathcal{N}(\Vp+v,\delta)}\int_{\Vp+v}|x-y|^{m-s}d\mu_{\Vp+v}(y)d\mu(x)d\Ha^m(v)d\munm(V)\\
&=\liminf_{\delta\to0} (2\delta)^{-m}\int\int\int_{\{v\in V: d(x,\Vp+v)\leq\delta\}}\int|x-y|^{m-s}d\mu_{\Vp+v}(y)d\Ha^m(v)d\mu(x)d\munm(V)\\ 
\end{align*}

The two inner integrals can be bounded applying (\ref{bound}) where $B=\{v\in V: d(x,\Vp+v)\leq\delta\}$ so that $P^{-1}_V(B)=\{y:|P_V(x-y)|\leq\delta\}$.
Hence we get,
\begin{align*}
&\int_{\Gh}\int_VI_{s-m}(\mu_{\Vp+v})d\Ha^m(v)d\munm(V)\\
&\leq\liminf_{\delta\to0}(2\delta)^{-m}\int\int \int_{\{y:|P_V(x-y)|\leq\delta\}}|x-y|^{m-s}d\mu (y) d\mu(x)d\munm\\
&=\liminf_{\delta\to0}(2\delta)^{-m}\int \int|x-y|^{m-s}\munm(\{V:|P_V(x-y)|\leq\delta\})d\mu(y)d\mu(x)
\end{align*}
Now we can apply equation (\ref{mubydelta}) to finally get 
\begin{align*}
&\int_{\Gh}\int_VI_{s-m}(\mu_{\Vp+v})d\Ha^m(v)d\munm(V)\\
& \leq c\int\int|x-y|^{-s}d\mu(y)d\mu(x)=c I_{s}(\mu).
\end{align*}
\end{proof}

The following lemma, establishes the dimension upper bound for Theorem \ref{isoslice}.
\begin{lemma}\label{allupperbound}
Let $A\subset\Rnn$ be a Borel set with $0<\Ha^s(A)<\infty$ for some $m<s\leq2n$. Then, for all $x\in\Rnn$,\[\dim[A\cap(\Vp+x)]\leq s-m\text{ for }\munm-a.e. V\in\Gh.\]
\end{lemma}

\begin{proof}
The proof begins with the following,

\textbf{Claim:} It suffices to show that  for every $x\in \R^{2n}$ and $r>0$,
\begin{equation}\label{minusball}
\int^*\Ha^{s-m}[(A\setminus B(x,r))\cap(\Vp+x)]d\munm(V)\lesssim r^{-m}\Ha^s(A\setminus B(x,r)).
\end{equation}
Indeed, if \ref{minusball} holds then for each $j\in\mathbb{N}$
\[\int^*\Ha^{s-m}[(A\setminus B(x,\frac{1}{j}))\cap(\Vp+x)]\lesssim \frac{1}{j}^{-m}\Ha^s(A)<\infty,\]
so that $\dim[(A\setminus B(x,\frac{1}{j}))\cap(\Vp+x)]\leq s-m$. Since \[A\cap(\Vp+x)=\cup_{j\in\mathbb{N}}(A\setminus B(x,\frac{1}{j}))\cap(\Vp+x),\]
by the countable stability of Hausdorff dimension the Lemma follows.

All is left is to prove (\ref{minusball}). This is done using the same argument as in \cite[Theorem 6.5]{Mattila4}.\\
First, by translating by $x$, we may assume $x=0$. For each $k\in\mathbb{N}$, pick balls $B_{kj}$ such that 
\begin{enumerate}[(i)]
\item $A\setminus B(0,r)\subset\cup_{j}B_{kj}\subset \Rnn\setminus B(0,r/2),$
\item diam$B_{kj}< \frac{1}{k},$ and
\item $\sum_{j=1}^{\infty}(\text{diam}B_{kj})^s<\Ha^s(A)+1.$
\end{enumerate}
Since $B_{kj}\subset\Rnn\setminus B(0,r/2)$, (\ref{mubydelta}) says that 
\[\munm(\{V:B_{kj}\cap \Vp\neq \varnothing\})\lesssim (\frac{r}{2})^{-m}(\text{diam}B_{kj})^{m}.\]
Since $s>m$ we get 
\[\int_{\Gh}\text{diam}(B_{kj}\cap \Vp)^{s-m}d\munm\lesssim (\frac{r}{2})^{-m}(\text{diam}B_{kj})^{s}.\] By the definition of $\Ha^s$ and Fatou's lemma,
\begin{align*}
\int^* &\Ha^{s-m}[(A\setminus B(0,r))\cap \Vp]d\munm(V)\\
&\leq \int^*\liminf_{k\to\infty}\sum_{j}\text{diam}(B_{kj}\cap \Vp)^{s-m}d\munm(V)\\
&\leq \liminf_{k\to\infty}\sum_j\int\text{diam}(B_{kj}\cap \Vp)^{s-m}d\munm(V)\\
&\lesssim r^{-m}\sum_{j}(\text{diam} B_{kj})^s\leq r^{-m}(\Ha^s(A)+1).
\end{align*}
This finishes this argument. 
\end{proof}

This aids the proof of the following result which will prove useful in the proof of the dimension lower bound in Theorem \ref{isoslice}.

\begin{theorem}\label{InSlice}
Let $A\subset\Rnn$ be a Borel set with $0<\Ha^s(A)<\infty$ for some $m<s\leq2n$. Then,
\[\dim[A\cap(\Vp+x)]=s-m, \text{ for }\Ha^s\times\munm-a.e.\ (x,V)\in A\times\Gh.\]
\end{theorem}
\begin{proof}
Suppose there is a $\Ha^s-$positive measure subset of $A$ such that for every $x$ in that subset
\[dim[A\cap(\Vp+x)]<s-m\ \text{for a }\munm-\text{positive measure subset of }\Gh.\]
By regularity of $\Ha^s$ we may assume such subset is compact. This is to say, there is some $m<\sigma<s$ and compact $F\subset A$ such that $\Ha^s(F)>0$ and $\forall\ x\in F$,
\[\munm(\{V:dim[A\cap(\Vp+x)]<\sigma-m\})>0.\]
By Frostman Lemma, there is $\mu\in\Ma(F)$ such that $\mu(B(x,r))\leq r^s$ for all $x\in\Rnn, r>0$. By Tonelli's theorem,
\begin{align*}
\int\mu(\{x:\dim &[A\cap(\Vp+x)]<\sigma-m\})d\munm(V)\\
&=\int\munm\{V:dim[A\cap(\Vp+x)]<\sigma-m\}d\mu(x)>0.
\end{align*} 
So there exist a compact set $G\subset \Gh$ such that $\munm(G)>0$ and for all $V\in G$,
\[\mu(\{x:\dim[A\cap(\Vp+x)]<\sigma-m\})>0.\]
By Lemma \ref{FrostAbs}, $P_{V\#}\mu\ll\Ha^m$ for $\munm-$almost every $V\in\Gh$, and since \[P_{V\#}\mu(\{v\in V: \dim[A\cap(\Vp+v)]<\sigma-m\})= \mu(\{x\in \Rnn: \dim[A\cap(\Vp+x)]<\sigma-m\})>0,\]

we get that for $V \in G$
\[\Ha^m(\{v\in V: \dim[A\cap(\Vp+x)]<\sigma-m\})>0.\]
Now, if $\mu_{\Vp+v}(\Rnn)>0$ with $\dim[A\cap(\Vp)+v]<\sigma-m$ one has $I_{\sigma-m}(\mu_{\Vp+v})=\infty$, so by (\ref{equal}) we get
\[\int_{\Gh}\int_V I_{\sigma-m}(\mu_{\Vp+v})d\Ha^m(v)d\munm(V)=\infty,\]
which contradicts Lemma \ref{10.7analogue}. Hence, it must be that for $\Ha^s-$almost every $x\in A$
\[\dim[A\cap(\Vp+x)]\geq s-m\text{ for } \munm-a.e. V\in\Gh.\]

Combining this with Lemma \ref{allupperbound}, we have that for $\Ha^s-$almost every $x\in A$
\[\dim[A\cap(\Vp+x)]=s-m, \text{ for }\munm-a.e.\ V\in \Gh.\]
By Tonelli's theorem, one may change the order of the measures at will to obtain that the dimension equality holds almost surely in the product measure space $A\times\Gh$.
\end{proof}

If one changes the roles of $V$ and $x$ in Lemma \ref{allupperbound}, a stronger result holds. That is,
\begin{equation}
\text{For all } V\in\Gh,\ \Ha^{s-m}[ A\cap(\Vp+x)]\leq\infty, \text{ for }\Ha^m-a.e.\ x\in V.
\end{equation}
This is simply a special case of Eilenberg's lemma (Theorem 13.3.1, \cite{BurZal})  which says that there is a constant $k=k(s,m)$ such that for every $m$-dimensional plane, isotropic or not,
\begin{equation}\label{eilen}
\int_{V}\Ha^{s-m}(A\cap (\Vp+v))d\Ha^m(v)\leq k(s,m)\Ha^{s}(A)
\end{equation}
so in particular, if $\Ha^s(A)<\infty,$
\[\Ha^{s-m}(A\cap (\Vp+v))<\infty,\text{ for }\Ha^m-a.e\ v\in V. \]
In view of \eqref{minusball} one could expect that the similar statement:
\[\text{For all } x\in\R^{2n},\ \Ha^{s-m}[ A\cap(\Vp+x)]\leq\infty, \text{ for }\munm-a.e.\ V\in\Gh,\]
holds. I have not been able to prove this, but the following weaker result holds.

\begin{theorem}\label{UpperBound}
Let $s\geq m$ and $A\in\Rnn$ be a Borel set with $\Ha^s(A)<\infty$. Then for $\Ha^{2n}-$almost every $x\in\Rnn$
\[\Ha^{s-m}(A\cap(\Vp+x))<\infty, \text{ for $\munm-$almost every } V\in\Gh\]
\end{theorem}
\begin{proof}\text{ }

\textbf{Claim} Fix any $R>0$, then for $\Ha^{2n}-$almost every $x\in\Rnn$
\[\Ha^{s-m}(A\cap(\Vp+x))\mathbbm{1}_{\mathcal{N}(V,R)}(x)<\infty,\]
for $\munm-$almost every $V\in\Gh$. Here, for a set $E\subset\Rnn$ and $r>0$, $\mathcal{N}(E,r):=\cup_{x\in E}B(x,r)$ is the tubular neighborhood, of radius $r$, of the set $E$.

To see this claim, let $\lambda$ be the measure on $\Rnn$ given by $d\lambda=\mathbbm{1}_{\mathcal{N}(V,R)}d\Ha^{2n}.$ In other words $\lambda=\Ha^{2n}\lfloor_{\mathcal{N}(V,R)}.$ Then there is a constant $\tilde{c}=\tilde{c}(R, m, n)$, independent of $V$, such that $P_{V\#}\lambda=c\Ha^m$. Indeed, since $\mathcal{N}(V,R)=P_{\Vp}^{-1}(B_{\Vp}(0,R))$, for any measurable set $S\subset V$ we have that 
\begin{align*}
P_{V\#}\lambda(S) &=\lambda(P_{V}^{-1}(S))\\
&=\Ha^{2n}(P_{V}^{-1}(S)\cap P_{\Vp}^{-1}(B_{\Vp}(0,R)))\\
&=\Ha^{2n}(S\times B_{\Vp}(0,R))\\
&=c\Ha^{2n-m}(B_{\Vp}(0,R))\Ha^m(S)=\tilde{c}(R, m, n)\Ha^m(S),
\end{align*}
where $\tilde{c}(R,m,n)$ is the  product of $\Ha^{2n-m}-$measure of the ball of radius $R$ in $\R^{2n-m}$ with the constant $c$ such that $\Ha^{2n}=c\Ha^{2n-m}\times\Ha^{m}$.  Now, by Eilenberg's lemma, Tonelli's theorem, and the fact that for any $x\in\Rnn$, $\Vp+x=\Vp+P_V(x)$, we compute
\begin{align*}
\int_{\Rnn} &\int_{\Gh}\Ha^{s-m}[A\cap(\Vp+x)]\mathbbm{1}_{\mathcal{N}(V,R)}(x)d\munm(V)d\Ha^{2n}(x)\\
&=\int_{\Gh} \int_{\Rnn}\Ha^{s-m}[A\cap(\Vp+x)]\mathbbm{1}_{\mathcal{N}(V,R)}(x)d\Ha^{2n}(x)d\munm(V)\\
&=\int_{\Gh} \int_{\Rnn}\Ha^{s-m}[A\cap(\Vp+P_V(x))]d\lambda(x)d\munm(V)\\
&=c(R, m, n)\int_{\Gh} \int_{V}\Ha^{s-m}[A\cap(\Vp+v)]d\Ha^m(v)d\munm(V)\\
&\leq c(R, m, n)k(s, m)\Ha^m(A)<\infty,
\end{align*}
where the last line follows by integrating (\ref{eilen}) over $\Gh$ with respect to the probability measure $\munm$. This proves the claim.\\

Now the theorem is proven by contradiction. Suppose there is a set $\Omega\subset\Rnn$ such that $\Ha^{2n}(\Omega)>0$, and for all $x\in\Omega$, $\Ha^{s-m}(A\cap(\Vp+x))=\infty,$ for a positive $\munm$ measure set of planes $V\in \Gh$. By inner regularity of $\Ha^{2n}$ we may assume, without loss of generality, that $\Omega$ is compact. Then there is $R_0>0$ such that $\Omega\subset B(0,R_0)$. Since $0\in V$ it follows that $B(0,R_0)\subset\mathcal{N}(V,R_0)$ for every $V\in \Gh$. Therefore, $\Omega\subset\mathcal{N}(V,R_0)$ for every $V\in \Gh$. Hence, for every $x\in\Omega$ \[\Ha^{s-m}[A\cap(\Vp+x)]\mathbbm{1}_{\mathcal{N}(V,R_0)}(x)=\Ha^{s-m}[A\cap(\Vp+x)]=\infty,\] for a positive $\munm$ measure set of planes $v\in\Gh$. This contradicts the claim and finishes the proof. 

\end{proof}

The main slicing result is proven next.

\begin{proof}[Proof of Theorem \ref{isoslice}] 
 In view of Lemma \ref{allupperbound}, we only need to check that all $x$ outside of a set of dimension less than $m$, the upped bound $\dim[A\cap(\Vp+x)]\geq s-m$ holds for a $\munm-$positive measure set of planes $V\in\Gh$.

Without loss of generality, by Borel regularity of $\Ha^s$ we may assume that $A$ is compact. Let $B=\{x\in\Rnn: \munm(V\in\Gh: \dim[A\cap(\Vp+x)]\geq s-m)=0\}$, that is to say, for all $x\in B$, $\dim[A\cap(\Vp+x)]< s-m, \text{ for }\munm-\text{a.e. }V\in\Gh.$ Since $A$ is compact, the function $(x,V)\mapsto \dim[A\cap(\Vp+x)]$ is a Borel function (see \cite{MattilaMauldin}), which shows the set $B$ is a Borel set. Suppose, to obtain a contradiction, that $\dim B>m$. Pick $\nu\in\Ma(B)$ with finite $m-$energy. Let $\mu=\Ha^s\lfloor_{A}\in\Ma(A)$, possibly restricted to a further subset, so that $\mu(B(x,r))\lesssim r^s$ for all $x\in\Rnn, r>0$. In particular, by Lemma \ref{FrostAbs}, $I_{m}(\mu)<\infty.$ By assumption, for $\nu-$almost every $y\in\Rnn$ $\dim[A\cap(\Vp+y)]< s-m$ for $\munm-$almost every $V\in\Gh$. Applying Tonelli's theorem to the measures, we get that for $\munm-$-almost every $V\in\Gh$
\begin{equation}\label{one}
\dim[A\cap(\Vp+y)]< s-m \text{ for }\nu-a.e.\ y\in\Rnn.
\end{equation}
Similarly, by Theorem \ref{InSlice}, for $\munm-$almost every $\V\in\Gh$ we have 
\begin{equation}\label{two}
\dim[A\cap(\Vp+x)]\geq s-m \text{ for }\mu-a.e.\ x\in\Rnn.
\end{equation}
To find a contradiction one tries to find $y\in spt(\nu)$, $x\in spt(\mu)$ satisfying (\ref{one}) and (\ref{two}) respectively and such that $\Vp+x=\Vp+y$, or equivalently, $P_V(x)=P_V(y)$.\\
As seen Lemma \ref{FrostAbs}, for $\munm-$almost every $V\in\Gh$, both $P_{V\#}\mu$ and $P_{V\#}\nu$ are absolutely continuous with respect to $\Ha^m$, with densities $\mu_V,\ \nu_V \in L^2(V)$.  For such $V$ define
\begin{align*}
&A_V:=\{x\in\Rnn: \dim[A\cap(\Vp+x)]\geq s-m\}\\
&B_V:=\{y\in\Rnn: \dim[A\cap(\Vp+y)]< s-m\}\\
&C_V:=\{v\in V: \mu_V(v)\nu_V(v)>0\}.
\end{align*}
For $\munm-$almost every $V$, $\mu(\Rnn\setminus A_V)=0$, and $\nu(\Rnn\setminus B_V)=0$. By (\ref{L1comp}), for a $\munm-$positive measure set of planes, $\Ha^m(C_V)>0$.  So we can pick $V\in\Gh$ for which all three things are simultaneously satisfied. By (\ref{equal}) we have 
\[\int_V\mu_{\Vp+v}(\Rnn\setminus A_V)d\Ha^m(v)=\mu(\Rnn\setminus A_V)=0,\]
and
\[\int_V\nu_{\Vp+v}(\Rnn\setminus B_V)d\Ha^m(v)=\nu(\Rnn\setminus B_V)=0.\]
By (\ref{ProjSlice}) we also have $0<\mu_V(v)=\mu_{\Vp+v}(\Rnn)$, and $0<\nu_V(v)=\nu_{\Vp+v}(\Rnn)$, so that $\mu_{\Vp+v}(A_V)>0$ and $\nu_{\Vp+v}(B_V)>0.$ This shows there is $v\in C_V$ such that $\mu_{\Vp+v},$ and $\nu_{\Vp+v}$ are both positive. Hence, we can find $x\in A_V$, $y\in B_V$, such that $P_V(x)=P_V(y)=v$ as needed.

\end{proof}


\section{Applications to the Heisenberg group }

Now the aim is to apply the previously discussed results to the geometry of the Heisenberg groups. These groups are very widely studied and there are many references, expository and otherwise, to geometry and analysis in these groups (see for instance \cite{CDPT}). Here we will cover the basics of the Heisenberg group together with those properties that will be relevant to the contents of this paper.  The $n^{th}$ Heisenberg group, denoted $\Hn$, is a nilpotent Lie group whose background manifold is $\Cn\times\R$ and whose Lie algebra has a step 2 stratification $\mathfrak{h}^n=V_1\oplus V_2$ where $V_1$ has dimension $2n$, $V_2$ has dimension 1 and satisfies $[V_1,V_1]=V_2$ with all other brackets being  trivial. The most common way to represent $\Hn$ is as the set $\Cn\times\R$ or $\Rnn\times \R$. We use this last two presentations interchangeably by identifying $\Cn$ with $\Rnn$ in the usual way (i.e. $z=x+iy$). The group law on $\Hn$ is then given by
\[(z,t)*(w,s)=(z+w,t+s-\frac{1}{2}\omega(z,w)).\]
With this representation, it is easy to see that $\Hn$ is an $\R$-bundle over $\Cn$ with bundle map $\pi:\Hn\to\Cn$ given by $\pi(z,t)=z$. The left invariant vector fields are given by
\[X_i=\frac{\partial}{\partial x_i}-\frac{x_{i+n}}{2}\frac{\partial}{\partial t},X_{i+n}=\frac{\partial}{\partial x_{i+n}}+\frac{x_{i}}{2}\frac{\partial}{\partial t}, \text{ for } i=1,\ldots,n , \text{ and } T=\frac{\partial}{\partial t}.\]

The first $2n$ vector fields form a basis for a bracket-generating sub-bundle of $\mathfrak{h}^n$. This allows us to use these vectors to create a well-defined distance in $\Hn$. We say a $\mathcal{C}^1$ curve $\gamma:I\to\Hn$ is \textit{horizontal}, or \textit{admissible}, if $\gamma'(\tau)\in Span\{X_1,\ldots, X_{2n}\},\ \forall \tau\in I$. That is to say, if $\gamma$ is horizontal then $\gamma'(\tau)=\sum_{j=1}^{2n}a_j(\tau)X_j$, hence we can define the length of $\gamma$ as $|\gamma|^2=\int_I\left[\sum_{j=1}^{2n}a_j^2(\tau)\right]^{1/2}d\tau$ The associated path distance, known as Carnot-Carath\'{e}odory distance, is given by
\[d_{cc}(p,q)=\inf\{|\gamma|:\gamma\in\mathcal{C}^1([0,1],\Hn)\text{  is horizontal, with } \gamma(0)=p\text{ and }\gamma(1)=q\}.\]
This makes $\Hn$ a sub-Riemannian manifold whose Hausdorff dimension is different from its topological dimension. In fact, $\Hn$ is Ahlfors $(2n+2)$-regular. More generally, for $A\subset\Hn$ its Hausdorff dimension with respect to the Heisenberg metric is greater than its dimension with respect to the Euclidean metric in $\R^{2n+1}$. In this section, $\dim A$ refers to the dimension with respect to the Heisenberg metric. To avoid confusion, if any reference to the Euclidean Hausdorff dimension of a set is needed, it will be denoted by $\dim_E A$.  The Heisenberg group also admits a gauge norm, known as the Kor\'{a}nyi norm, which induces a metric that is bi-Lipschitz equivalent to $d_{cc}$.  Given that Hausdorff dimension is invariant under bi-Lipschitz maps, and due to the explicit formula for the Kor\'{a}nyi norm, the Kor\'{a}nyi metric is often used in interchagebly with the Carnot-Caratheodory distance when exploring questions related to dimension of sets. The norm is given by
\[||(z,t)||_{\He}^4=|z|^4+16t^2,\]
and the metric, by
\[d_{\He}(p,q)=||q^{-1}p||_{\He},\]
where $q^{-1}$ is the group inverse of $q$ (which happens to be $-q$). 

One of the most important properties of $\Hn$ is that it admits a homogeneous structure given by the homogeneous dilations
\[\delta_r(z,t)=(rz,r^2t), r>0.\]
The map $\delta_r$ is an automorphism of the group that dilates distances by a factor of $r$. These dilations, combined with the group structure, make $\Hn$ into a ``almost vector space" of sorts. In fact, the Heisenberg group is the simplest example of a much larger class of groups, known as Carnot groups, which share many of these properties.  With this homogeneous structure it makes sense to to talk about ``vector subspaces", subgroups of $\He$ that are closed under homogeneous dilations. These are known as \textit{homogeneous subgroups} and in $\Hn$ they come in 2 types, those that intersect the $t$ axis trivially and those that contain the whole $t-axis$. The former subgroups are called \textit{horizontal subgroups}, while the latter are referred to as \textit{vertical subgroups}. Suppose that $\V\subset\Hn$ is a horizontal subgroup. Then for $(z,0),(w,0)\in\V$, their product $(z,0)*(w,0)=(z+w,-\frac{1}{2}2\omega(z,w))$ must also be in $\V$, so that $\omega(z,w)=0$. That is to say, $\V\subset\Cn\times\{0\}$ must be contained in $V\times\{0\}$ where $V$ is an isotropic subspace of $\Cn=\Rnn$. Indeed, there is a one to one correspondence between isotropic subspaces of $\Rnn$ and horizontal subgroups of $\Hn$. The notation $\V$ is used to denote the horizontal subgroup corresponding to the plane $V$. Moreover, if we denote by $\V^{\perp}$ the Euclidean orthogonal complement of $\V$, one can easily check that $\V^{\perp}$ is a vertical subgroup of $\Hn$. For each $V\in \Gh$, $\Hn$ admits a semi-direct splitting $\Hn=\V^{\perp}\rtimes\V$. This semi-direct splitting gives raise to well-defined projection maps $P_{\V}:\Hn\to\V$ and $P_{\V^{\perp}}:\Hn\to\V^{\perp}$ defined by ``reading off" the respective component of a point $p\in\Hn$. Horizontal projections are simply Euclidean orthogonal projections onto the corresponding isotropic subspace, whereas vertical projections are given by the formula $P_{\V^{\perp}}(p)=p*P_{\V}(p)^{-1}$.

The relation between horizontal subgroups and isotropic subspaces of $\Hn$ allows to obtain strong results about Heisenberg dimension distortion by horizontal projections. A result along this lines was first observed by Z.Balogh, E. Durand-Cartagena, K. F\"{a}ssler, P. Mattila, and J.T. Tyson in \cite{BDCFMT}, where they proved the following,

\begin{theorem}
Let $A\subset\Hn$ be a Borel set of dimension $s$ with $0<\Ha^s(A)<\infty$. Then, 
\begin{enumerate}
\item If $2\leq s\leq m+2$, $\dim P_{\V}A \geq s-2$ for $\munm$-a.e. $V\in \Gh$.
\item If $s> m+2$, $\Ha^m( P_{\V}A) > 0$ for $\munm$-a.e. $V\in \Gh$.
\end{enumerate} 
\end{theorem}
This theorem is a consequence of Theorem \ref{ISOPROJ} combined with the fact that the bundle map $\pi:\Hn\to\Cn$ does not increase Hausdorff dimension of sets, and it decreases it by at most 2 (this is sharp). We can readily see how this, together with part (3) of Theorem \ref{ISOPROJ}, implies the following,
\begin{corollary}
\item If $s> 2m+2$, $Int( P_{\V}A) \neq\varnothing$ for $\munm$-a.e. $V\in \Gh$.
\end{corollary}
Vertical projections, on the other hand, are not Euclidean projections. In fact, the best regularity of the map $P_{\V}$ is (locally) $\frac{1}{2}$-Holder. Therefore, analyzing dimension distortion by these maps is a much more complicated task that does not come as a corollary of any known projection theorems in Euclidean space. In \cite{BFMT}, the authors obtained dimension distortion estimates for these vertical projections and conjectured what seems like feasible  sharp bounds for said distortion. Later in \cite{FasslerHovila} K. F\"{a}ssler and R. Hovila improved the estimates for $\He^1$, and recently T. Harris improved the $
\He^1$ bounds even further in \cite{THarris}.

 Here, with the results obtained for $\Gh$ one can easily prove Theorem \ref{BFMT} which in turned will help prove Theorem \ref{sliceinh}.

\begin{reptheorem}{BFMT}
Let $A,B\subset\Hn$ be Borel sets, and let $\V$ denote the horizontal subgroup corresponding to the isotropic plane $V\in\Gh$.
\begin{enumerate}
\item If $\dim A,\ \dim B>m+2$ then,
\[\munm(V\in\Gh:\Ha^m(P_{\V}A\cap P_{\V}B)>0)>0.\]
\item If $\dim A,\ \dim B>2m+2$ then,
\[\munm(V\in \Gh:Int(P_{\V}A\cap P_{\V}B)\neq\varnothing)>0.\]
\item If  $\dim A\geq m+2,\ \dim B\leq2m+2$ but $\dim A + \dim B>2m+4$, then
\[ \munm(V\in \Gh: \dim  (P_VA\cap P_VB)>\dim  B -2 -\epsilon)>0.\]
\end{enumerate}
\end{reptheorem}
\begin{proof}
One can easily check that for any horizontal subgroup $\V$, $d_{\Hn}\lfloor_{\V}=d_{E}\lfloor_{\V}$, therefore for any set $D\subset \Hn$ $\dim P_{\V}D=\dim_{E}P_{\V}D$. Moreover, $P_{\V}=P_V\circ\pi$ and it was shown in \cite{BFMT} that $\dim_E\pi(D)\geq\dim D-2$. Therefore, the theorem is proven by simply applying Theorem \ref{isoslice} to the sets $\pi(A), \pi(B) \subset\Rnn\times\{0\}\subset\Hn$.
\end{proof}
Hidden within Theorem \ref{BFMT} is the following Lemma that will be useful later,

\begin{lemma}\label{MutualH}
If $s,\sigma\in\R$ are such that $m+2<s,\sigma \leq 2n+2$ and $\mu,\nu\in\Ma(\Hn)$ satisfy \[\mu(B_{\Hn}(p,r))\lesssim r^s,\text{ and }\nu(B_{\Hn}(p,r))\lesssim r^{\sigma},\]
then for $\munm-$almost every $\V\in\Gh$, $P_{\V\#}\mu$ and $P_{V\#}\nu$ have densities $\mu_V$ and $\nu_V$ respectively and for a $\munm-$positive measure set of planes $V\in\Gh$,
\[\Ha^m(\{v\in\V:\mu_V(v)\nu_V(v)>0\})>0\]
\end{lemma}
\begin{proof}
The fact that for almost every $V\in\Gh$  $P_{\V\#}\mu$ and $P_{\V\#}\mu$ have densities is precisely the content of  \cite[Proposition 6.1]{BFMT}. As before, for each $\V$, $P_{\V}=P_V\circ\pi$, so that $P_{\V\#}\mu=P_{V\#}\pi_{\#}\mu$. Hence, the aim is to show that both $I^E_m(\pi_{\#}\mu)$ and $I^E_m(\pi_{\#}\nu)$ are finite, where $I^E_m$ denotes the Euclidean m energy. The lemma will then follow from (\ref{holderen}) and the computation in (\ref{L1comp}).\\
It is enough to show that if $\beta>m+2$ and $\eta\in\Ma(\Hn)$ satisfies $\eta(B_{\Hn}(p,r))\lesssim r^{\beta}$ for all $p\in\Hn$ and $r>0$, then $I^E_m(\pi_{\#}\eta)<\infty$. First note that by definition,
\[I^E_m(\pi_{\#}\eta)=\int\int|\pi(p)-\pi(q)|^{-m}d\eta(q)d\eta(p).\]
We will bound the inner integral independently of $p$ and use the fact that $\eta(\Hn)<\infty$ to bound the double integral. Since $\eta$ is compactly supported, there is $R_0>0$ such that $spt(\eta)\subset B_E(0,R_0)$. For each fixed $z\in\Rnn$, $\{q\in\Hn:|\pi(q)-z|\leq r\}$ is a cylinder of (Euclidean) radius $r$ over the ball $B^{2n}_E(z,r)$. Therefore
\[\{q\in\Hn:|\pi(q)-z|\leq r\}\subset B_E^{2n}(z,r)\times [-R_0,R_0].\]
By Lemma 6.5 in \cite{BFMT} one can choose $N\lesssim r^{-2}$ Heisenberg balls, $\{B_{\Hn}^j\}_{j=1}^N$ of radius $r$, such that $B_E^{2n}(z,r)\times [-R_0,R_0]\subset\cup_j B^j_{\Hn}$. It follows from the properties of $\eta$ that 
\[\eta([\{q\in\Hn:|\pi(q)-z|\leq r\})\lesssim r^{\beta-2}.\]
By theorem 1.15 in \cite{mattila2} we can compute
\begin{align*}
\int|\pi(q)-z|^{-m}d\eta(q)&=\int_0^{\infty}\eta(\{q\in\Hn:|\pi(q)-z|\leq r^{-1/m}\})dr\\
&=\int_0^1\eta(\{q\in\Hn:|\pi(q)-z|\leq r^{-1/m}\})dr+\int_1^{\infty}\eta(\{q\in\Hn:|\pi(q)-z|\leq r^{-1/m}\})dr\\
&\lesssim \eta(\Hn)+\int_1^{\infty}r^{-\frac{\beta-2}{m}}dr<\infty.
\end{align*}
Where the last line is finite because $\beta-2>m$. The lemma follows.
\end{proof}

The projection theorem also help us obtain our claimed results for exceptional sets for the slicing theorem in $\Hn$. Vertical subgroups are normal subgroups of $\Hn$, its right cosets $\V^{\perp}*p$, for $p\in\Hn$, form a partition of $\Hn$.  Just as in the Euclidean case, given a measure $\mu\in\Ma(\Hn)$ one can define the sliced measure $\mu_{\V^{\perp}*p}$ for each $\V$ and $p\in \V$. A detailed construction of these sliced measures can be found in \cite{BFMT}, but it follows the same scheme as the construction of sliced measures by planes in Euclidean space. Here we will simply state the properties that will be relevant to us. Firstly, as expected from ``sliced'' measures, $spt(\mu_{\V^{\perp}*p})\subset spt(\mu)\cap\V^{\perp}*p$. One can, however, say a lot more than that. In fact, for any Borel set $\Omega\subset\V$ one has
\begin{equation}\label{SlicedIntegral}
\int_{\Omega}\mu_{\V^{\perp}*p}(\Hn)d\Ha^m(p)\leq\mu({P_{\V}^{-1}\Omega}),
\end{equation}
with equality whenever $P_{\V\#}\mu\ll \Ha^m$. In particular, in this case if $\Omega=\V$ we get
\begin{equation}\label{totalmassh}
\int_{\V}\mu_{\V^{\perp}*p}(\Hn)d\Ha^m(p)=\mu(\Hn).
\end{equation}
 Whenever $P_{\V\#}\mu\ll \Ha^m$, in parallel with equation (\ref{SlicedIntegral}) one has,
\[\mu({P_{\V}^{-1}\Omega})=P_{\V\#}\mu (\Omega)=\int_{\Omega}P_{\V\#}\mu (p)d\Ha^m(p).\]
Hence we also have 
\begin{equation}\label{pushslice}
P_{\V\#}\mu(p)=\mu_{\V^{\perp}*p}(\Hn).
\end{equation}

In \cite{BFMT}, the authors proved the Heisenberg group analogue of Theorem \ref{InSlice}. This will be used in the proof of Theorem \ref{sliceinh} so it is stated here without proof.

\begin{theorem}\label{InSliceH}
Let $A\subset\Hn$ be a Borel set with $0<\Ha^s(A)<\infty$ for some $m+2<s\leq2n+2$. Then,
\[\dim[A\cap(\V^{\perp}*p)]=s-m, \text{ for }\Ha^s\times\munm-a.e.\ (x,V)\in A\times\Gh.\]
\end{theorem}
 
 The argument used to prove Theorem \ref{InSlice} is just a simple modification of the argument used in \cite{BFMT} to prove Theorem \ref{InSliceH}. In particular, as a consequence of the proof of Theorem \ref{InSliceH} one obtains the analogue of Lemma \ref{allupperbound}.
 \begin{lemma}\label{allupperboundH}
 Let $A\subset\Hn$ be a Borel set with $0<\Ha^s(A)<\infty$ for some $m+2<s\leq2n+2$. Then, for all $p\in\Hn$
 \[\dim[{A\cap(\V^{\perp}*p)}]\leq s-m\text{ for }\munm-a.e.\ V\in\Gh.\]
 \end{lemma}

Just as before with Lemma \ref{allupperbound}, Lemma \ref{allupperboundH} holds at the level of dimension. One might expect that the stronger statement, at the level of measures, holds true. The following weaker result, analogous to Theorem \ref{UpperBound}, holds.

\begin{theorem}
Let $m+2<s\leq 2n$. If $A\subset\Hn$ is a Borel set such that $0<\Ha^s(A)<\infty$, then for almost every $p\in\Hn$
\[\Ha^{s-m}[A\cap(\V^{\perp}*p)]\leq s-m\text{ for }\munm-a.e.\ V\in\Gh.\]
\end{theorem}
\begin{proof}

First, we show that for any $R>0$, for $\mathcal{L}^{2n+1}-$almost every $p\in\Hn$
\begin{equation}\label{upperb}
\Ha^{s-m}[A\cap(\V^{\perp}*p)]\mathbbm{1}_{N_E(\V,R)}(p)<\infty,
\end{equation}
for $\munm-$almost every $V\in\Gh$. Here, as before, $N_E(\V,R)$ is the Euclidean tubular neighborhood of $\V$ of radius $R$ and $\mathcal{L}^{2n+1}$ is the $(2n+1)$-dimensional Lebesgue measure. We remark that there is a constant $d=d(n)$, depending only on $n$, such that $\mathcal{L}^{2n+1}=d\Ha^{2n+2}$. So using $\mathcal{L}^{2n+1}$ is equivalent to using the, arguably more intrinsic, measure $\Ha^{2n+2}$. \\
To show (\ref{upperb}), let $\lambda=\mathcal{L}^{2n+1}\lfloor_{N_E(\V,R)}$. Since $N_E(\V,R)=B_{\Vp,E}(0,R)\times\V$, for any measurable set $S\subset\V$ we have 
\begin{align*}
P_{\V\#}\lambda(S)&=\lambda(P_{\V}^{-1}(S))\\
&=\mathcal{L}^{2n+1}(P_{\V}^{-1}(S)\cap N_E(\V,R))\\\
&=\mathcal{L}^{2n+1-m}(B_{\Vp,E}(0,R))\mathcal{L}^m(S)\\
&=\tilde{\tilde{C}}(R,n,m)\Ha^m(S).
\end{align*}
Here $\tilde{\tilde{C}}(R,n,m)$ is a constant, depending only on $R$, $n$, and $m$, which involves the $\mathcal{L}^{2n+1-m}-$volume of the ball of radius $R$ in $\R^{2n+1-m}$, as well as the universal constant $c$ such that $c\Ha^m=\mathcal{L}^m$ in $\R^m$.\\
For each $\V$, the map $P_{\V}$ is 1-Lipschitz, so we can apply Elienberg's lemma to get 
\begin{align*}
&\int_{\Hn}\int_{\Gh}\Ha^{s-m}[A\cap(\V^{\perp}*p)]\mathbbm{1}_{N_E(\V,R)}(p)d\munm(V)d\mathcal{L}^{2n+1}(p)\\
&=\int_{\Gh}\int_{\Hn}\Ha^{s-m}[A\cap(\V^{\perp}*P_{\V}(p))]d\lambda(p)d\munm(V)\\
&=\tilde{\tilde{C}}(R,m,n)\int_{\Gh}\int_{\V}\Ha^{s-m}[A\cap(\V^{\perp}*v)]d\Ha^m(v)d\munm(V)\\
&\lesssim \Ha^{s}(A)<\infty,
\end{align*}
where the last line follows by integrating Elinebger's inequality with respect to $\munm$.\\
Now supppose there is $E\subset\Hn$ is such that $\mathcal{L}^{2n+1}(E)>0$ and for every $p\in E$,
\[\Ha^{s-m}[A\cap(\V^{\perp}*p)]=\infty,\]
for a $\munm-$positive measure subset of $\Gh$. Without lost of generality we may assume $E$ is compact, so there is $R_0>0$ such that $E\subset B_E(0,R_0)\subset N_E(\V,R_0)$ for every $\V$. So we have that for every $p\in E$
\[\Ha^{s-m}[A\cap(\V^{\perp}*p)]\mathbbm{1}_{N_E(\V,R_0)}(p)=\Ha^{s-m}[A\cap(\V^{\perp}*P_{\V}(p))]=\infty,\] 
 for a $\munm-$positive measure subset of $\Gh$ which is a contradiction. It follows that for $\mathcal{L}^{2n+1}-$almost every $p\in\Hn$
 \[\Ha^{s-m}[A\cap(\V^{\perp}*p]<\infty\text{ for }\munm-a.e. V\in\Gh.\]
\end{proof}

Finally, all is set to prove Theorem \ref{sliceinh}. The proof which follows the same technique as the proof of Theorem \ref{isoprojint}.

\begin{proof}[Proof of Theorem \ref{sliceinh}]
Without loss of generality assume $A$ is compact
and set
\[B:=\{p\in\Hn:\munm(V:\dim [A\cap(\V^{\perp}*p)]\geq s-m)=0\}.\]
Since $A$ is compact, $(p,\V)\to \dim [A\cap(\V^{\perp}*p)]$ is a Borel function, so $B$ is a Borel set. Suppose $\dim B>\sigma>m+2$ and pick $\nu\in\Ma(B)$ so that $\nu(B_{\Hn}(p,r))\lesssim r^{\sigma}$ so that by \cite[Proposition 6.1]{BFMT}, $P_{\V\#}\nu\ll\Ha^m$ for $\munm-$almost every $V\in\Gh$. Similarly let $\mu=\Ha^s\lfloor_{A}\in\Ma(A)$, possibly restricted to a further subset, so that $\mu(B_{\Hn}(p,r))\lesssim r^s$. It is also true that $P_{\V\#}\mu\ll\Ha^m$ for $\munm-$almost every $V\in\Gh$. By Theorem \ref{InSliceH}, for $\mu-$almost all $p\in\Hn$
\begin{equation}
\dim [A\cap(\V^{\perp}*p)]\geq s-m, \text{ for }\munm-a.e.\ V\in\Gh.
\end{equation}
By assumption, for $\nu-$almost every $q\in\Hn$
\begin{equation}
\dim [A\cap(\V^{\perp}*q)]< s-m, \text{ for }\munm-a.e.\ V\in\Gh.
\end{equation}
Applying Tonelli's theorem we may switch the order of the measures so that for $\munm-$almost every $V\in\Gh$
\begin{equation}\label{p}
\dim [A\cap(\V^{\perp}*p)]\geq s-m, \text{ for }\mu-a.e.\ p\in\Hn,
\end{equation}
and,
\begin{equation}\label{q}
\dim [A\cap(\V^{\perp}*q)]< s-m, \text{ for }\nu-a.e.\ q\in\Hn.
\end{equation}
As before, a contradiction is found by finding $p\in spt(\mu)$ and $q\in spt(\nu)$ satisfying (\ref{p}) and (\ref{q}) respectively, and such that $P_{\V}(p)=P_{\V}(q).$
For $\munm-$almost every $V\in\Gh$, $P_{\V\#}\mu$ and  $P_{\V\#}\nu$ have densities $\mu_V$ and $\nu_V$ respectively. For such $V$ define,

\begin{align*}
&A_{\V}:=\{p\in\Hn: \dim [A\cap(\V^{\perp}*p)]\geq s-m\}\\
&B_{\V}:=\{q\in\Hn: \dim [A\cap(\V^{\perp}*q)]< s-m\}\\
&C_{\V}:=\{v\in\V: \mu_V(v)\nu_V(v)>0\}.
\end{align*}
For $\munm-$almost every $\V$, $\mu(\Hn\setminus A_V)=0$ and $\nu(\Hn\setminus B_V)=0.$ By Lemma \ref{MutualH}, for a $\munm-$positive measure subset of $\Gh$, $\Ha^m(C_V)>0$. We can pick a horizontal subgroup where all 3 things are satisfied simultaneously. By \ref{totalmassh},
\[\int_{\Hn}\mu_{\V^{\perp}*v}(\Hn\setminus A_V)d\Ha^m(v)=\mu(\Hn\setminus A_V)=0,\]
and
\[\int_{\Hn}\nu_{\V^{\perp}*v}(\Hn\setminus B_V)d\Ha^m(v)=\nu(\Hn\setminus B_V)=0.\]
By (\ref{pushslice}), $0<\mu_V(v)=\mu_{\V^{\perp}*v}(\Hn)$, and $0<\nu_V(v)=\nu_{\V^{\perp}*v}(\Hn).$ Hence, $\mu_{\V^{\perp}*v}(A_V)>0$ and $\nu_{\V^{\perp}*v}(B_V)>0$ for almost every $v\in\V$. In particular, there is $v\in C_V$ for which both $\mu_{\V^{\perp}*v}$ and $\nu_{\V^{\perp}*v}$ are positive, so we find $p\in A_V$, $q\in B_V$ such that $P_{\V}(p)=P_{\V}(q)=v.$ This completes the proof.

\end{proof}

\bibliographystyle{plain}

\bibliography{BiblioProj}

\end{document}